\documentclass[12pt,a4paper]{amsart}
\usepackage[top=3.5truecm, bottom=4.0truecm, left=3.5truecm, right=3.5truecm]{geometry}

\usepackage{mathtools}
\usepackage{amsmath , amssymb}
\usepackage{amsfonts}
\usepackage{amsthm}
\usepackage{graphicx}
\usepackage{subcaption}
\usepackage{tikz}
\usepackage{tikz-cd}
\usepackage{centernot}
\usepackage{here}

\DeclareMathOperator\im{Im}

\DeclareMathOperator\lk{lk}
\DeclareMathOperator\Hom{Hom}

\theoremstyle{definition}
\newtheorem{definition}{Definition}[section]
\newtheorem*{definition*}{Definition}

\newtheorem*{theorem*}{Theorem}
\newtheorem{proposition}[definition]{Proposition}
\newtheorem*{proposition*}{Proposition}
\newtheorem{corollary}[definition]{Corollary}
\newtheorem{lemma}[definition]{Lemma}
\newtheorem*{lemma*}{Lemma}

\newtheorem*{Note*}{Note}

\newtheorem*{example*}{Example}
\newtheorem*{question*}{Question}
\newtheorem*{conjecture*}{Conjecture}

\numberwithin{equation}{section}

\begin{document}

\title[A PROPERTY OF $\gamma$ CURVES]{A PROPERTY OF $\gamma$ CURVES DERIVED FROM TOPOLOGY OF RELATIVE TRISECTED 4-MANIFOLDS}
\author{HOKUTO TANIMOTO}
\keywords{trisection; relative trisection; Homology; Intersection form; Torelli group}
\thanks{{\it Data sharing not applicable to this article as no datasets were generated or analysed during the current study.}}
\date{}
\maketitle
\vspace{-2ex}
\begin{abstract}
Any two $(g; 0, 0, k)$-trisected closed 4-manifolds are related to each other by cutting and regluing in regard to $\gamma$ curves by an element of the Torelli group. Moreover, Lambert-Cole showed that this element can be replaced by an element of the Johnson kernel. We consider similar situation in the relative case. Furthermore, we relate this situation to cork twist.
\end{abstract}

\section{Introduction}
A trisection was introduced by Gay and Kirby in \cite{GK1} as a decomposition of compact, connected 4-manifolds that consists of three 1-handlebodies. Castro, Gay and Pinz\'on-Caicedo dealt with a relative trisection in \cite{CGPC1} and \cite{CGPC2} more precisely.
Regarding a trisected 4-manifold, its homology and intersection form can be expressed in terms of its trisection diagram in the way shown by Feller, Klug, Schirmer and Zemke in \cite{FKSZ1}, and the way by Florens and Moussard in \cite{FM1}. Feller, Klug, Schirmer and Zemke also showed that closed $(g; 0, 0, k)$-trisections admit ``algebraically trivial'' diagrams, and thus any two $(g; 0, 0, k)$-trisected closed 4-manifolds are related to each other by cutting and regluing in regard to $\gamma$ curves of diagrams by an element of the Torelli group of the central surface. Moreover, Lambert-Cole \cite{L-C1} proved an analogy of Morita's result \cite{Mor1} for closed trisections, that is, an element of the Torelli group can be replaced by an element of the Johnson kernel.

In this paper, we consider similar situation regarding the relative trisected 4-manifolds. $(g; 0, 0, k)$-closed trisections are equivalent to $(g, (l, l, k); p, b)$-relative trisections, where $l = 2p+b-1$. It is relatively easy to show that $(g; 0, 0, k)$-trisections admit algebraically trivial diagrams, but even this triviality of diagrams is complicated in the relative case for several reasons. For example, homology classes of $\alpha$-, $\beta$-curves may not form a basis of $H_1(\Sigma)$. We show that the triviality of relative trisection diagrams also holds under some topological conditions.

This paper is organized as follows. In Section 2, we review definitions regarding relative trisections and their diagrams, and then state the main algebraic topological conditions of 4-manifolds for the triviality. In Section 3, we consider several facts which relate foundations of relative trisection and a proof of the main result. In Section 4, we prove the main theorem. In Section 5, we introduce corollaries related to the special relative trisection diagram and the Torelli group.

\subsection*{Acknowledgements}
The author would like to thank Hisaaki Endo for many discussions and encouragements. This work was supported by JST SPRING, Grant Number JPMJSP2106.

\section{Main results}
Trisection is a 4-dimensional analogy of Heegaard splitting. In relative trisection, sutured Heegaard splitting plays an important role to define and understand it.

Let $\Sigma_{g,b}$ be a connected, oriented genus $g$ surface with $b$ boundary components. We call a collection of $g-p$ disjoint simple closed curves in $\Sigma_{g,b}$ which cuts $\Sigma_{g,b}$ into a genus $p$ surface with $b$ boundary components {\it $(g,p)$-cut system of curves}. A {\it relative compression body} $C_{g,p;b}$, or simply {\it compression body}, is a cobordism from $\Sigma_{g,b}$ to $\Sigma_{p,b}$ which is constructed using only 2-handles. When a compression body is constructed by attaching 2-handles along cut system $\mu$ on a surface $\Sigma$, we also describe it as $C_{\mu}$. Furthermore, $\Sigma_{\mu}$ denotes the surface which is  cobordant to $\Sigma$ by $C_{\mu}$. Let $C$ be a compression body from $\Sigma$ to $B$. A {\it complete disk system} $D$ for $C$ is a disjoint union of disks $(D, \partial D) \subset (C, \Sigma)$ such that what is obtained by cutting $C$ along $D$ is homeomorphic to $B \times I$. Note that $\partial D$ is a cut system on $\Sigma$. Let $M$ be a compact, connected, oriented 3-manifold. A {\it sutured Heegaard splitting of $M$} is a decomposition such as 
\[ M = C_{g,p;b} \cup_{\phi} (-C_{g,p';b}), \]
where $\phi$ is an orientation-preserving diffeomorphism of $\Sigma_{g,b}$. Two sutured Heegaard splittings are {\it equivalent} if there exists a self-diffeomorphism of $M$ which can be split into two diffeomrphisms between compression bodies. If there exists a sutured Heegaard splitting of $M$, $M$ is a cobordism between surfaces $B$ and $B'$ rel $\partial$. Therefore, we can describe a sutured Heegaard splitting of $(M; B, B')$ as $(C, C')$ by using compression bodies $C$ and $C'$. Moreover, since each $C$ is diffeomorphic to $C_{\mu}$ for a cut system of curves $\mu$ on a surface $\Sigma$, this sutured Heegaard splitting can be described as follows;
\[ M = C_{\mu} \cup_{\Sigma} (-C_{\nu}). \]
When there exists such a decomposition of $M$, we call a 3-tuple $(\Sigma; \mu, \nu)$ {\it sutured Heegaard diagram of $M$}. For two diagrams $(\Sigma; \mu, \nu)$ and $(\Sigma; \mu', \nu')$, $(\Sigma; \mu, \nu)$ is  {\it diffeomorphism and handle slide equivalent} to $(\Sigma; \mu', \nu')$ if they are related by a diffeomorphism between $\Sigma$ and a sequence of handle slides within each $\mu$ and $\nu$. Two sutured Heegaard splitting are equivalent if and only if their diagrams are diffeomorphism and handle slide equivalent. Therefore, we may call simply the latter equivalent.
\vspace{0.5ex}

Let $X$ be a compact, connected, oriented, smooth 4-manifold with connected boundary. We give the definition of relative trisection according to \cite{Tan1}. Let  $g, k= (k_1, k_2, k_3), p, b$ be integers with $g \ge p \ge 0$, $b \ge 1$ and $g+p+b-1 \ge k_i \ge l$.
\begin{definition}
A $(g,k;p,b)$-{\it relative trisection} of $X$ is a decomposition $X = X_1 \cup X_2 \cup X_3$ such that:
\begin{enumerate}
\item[i)] $\Sigma = X_1 \cap X_2 \cap X_3$ is diffeomorphic to $\Sigma_{g,b}$;
   \vspace{0.5ex}
\item[ii)] $X_i \cap X_j = \partial X_i \cap \partial X_j$ for $i \neq j$, and each of them is diffeomorphic to $C_{g,p;b}$;
   \vspace{0.5ex}
\item[iii)] $X_i$ is diffeomorphic to $\natural^{k_i} S^1 \times D^3$;
   \vspace{0.5ex}
\item[iv)] $X_i \cap \partial X$ is diffeomorphic to $\Sigma_{p,b} \times I$;
   \vspace{0.5ex}
\item[v)] $(X_i \cap X_{i+1}) \cup_{\Sigma} -(X_{i-1} \cap X_i)$ is a sutured Heegaard splitting of $\Sigma_{p,b} \times I \, \sharp \, ( \sharp^{k_i-l} S^1 \times S^2)$.
\end{enumerate}
\end{definition}
\noindent Note that we orient each compression body $X_i \cap X_{i+1}$ as a submanifold of $\partial X_i$,  surfaces $\Sigma$ and $(X_i \cap X_{i+1}) \cap \partial X$ as submanifolds of $\partial (X_i \cap X_{i+1})$. 

We can obtain sutured Heegaard diagrams of $\Sigma_{p,b} \times I \, \sharp \, ( \sharp^{k_i-l} S^1 \times S^2)$ from the condition v). We define relative trisection diagram by using them.

\begin{definition}
A  {\it $(g,k;p,b)$-relative trisection diagram} $(\Sigma; \alpha^1, \alpha^2, \alpha^3)$ is a 4-tuple such that each triple $(\Sigma; \alpha^{i+1}, \alpha^i)$ is a sutured Heegaard diagram of $\Sigma_{p,b} \times I \, \sharp \, ( \sharp^{k_i-l} S^1 \times S^2)$.
\end{definition}
\noindent We also describe a diagram as $(\Sigma; \alpha, \beta, \gamma)$ satisfying $\alpha^1 = \alpha$, $\alpha^2 = \beta$, and $\alpha^3 = \gamma$. A diagram determines a unique (up to diffeomorphism) trisected 4-manifold with connected boundary such that $\Sigma$ and $C_{\alpha^i}$ are diffeomorphic to $X_1 \cap X_2 \cap X_3$ and $X_{i-1} \cap X_i$, respectively. Note that we also describe $C_{\alpha^i}$ and $\Sigma_{\alpha^i}$ as $C_i$ and $\Sigma_i$, respectively.
\vspace{0.5ex}

We prepare to state the algebraic triviality of relative trisection diagram. A genus $g$ sutured Heegaard splitting of  $\Sigma_{p,b} \times I \, \sharp \, ( \sharp^{k-l} S^1 \times S^2)$ is {\it standard} if there exists its diagram $(\Sigma; \mu, \nu)$ whose cut systems satisfy 
\begin{align}
i(\mu_i, \nu_j)  = \delta_{ij} \quad & (1 \le i, j \le g-p-k+l), \notag \\ 
\mu_i  = \nu_i \quad & ( g-p-k+l < i ), \notag
\end{align}
where $i(-, -)$ is the geometric intersection number on $\Sigma$. Even if a genus $g$ splitting of $\Sigma_{p,b} \times I$ is standard, $\{[\mu_i], [\nu_i]\}$ may not be a basis of $H_1(\Sigma_{g,b})$. Therefore, even in the case of $(g, (l, l, k); p, b)$-relative trisection, $[\gamma_i]$ may not be expressed as a linear combination of $[\alpha_j], \, [\beta_j]$.

There exist components of $\gamma$ which relate topology of $\partial X$. A relative trisection of $X$ induces an open book decomposition of $\partial X$ whose binder is $\Sigma_1$. Let $\psi : \Sigma_1 \to \Sigma_1$ be its monodromy. We describe a representation matrix of monodromy action $\xi_{\psi} : H_1(\Sigma_1, \partial \Sigma_1) \to H_1(\Sigma_1)$ for an basis of arcs as $A_{\phi}$, where $\xi_{\psi} ([a]) = [\psi (a) - a]$.

We state the main theorem of this paper. Let $d_i$ be $1$ if $1 \le i \le l$, be $0$ otherwise. The following theorem is a relative version of Theorem 4.4 in \cite{FKSZ1}.

\begin{theorem*}
Suppose $H_1(X) = H_1(\partial X) = 0$, and $X$ has a $(g, (l, l, k); p, b)$-relative trisection $\mathcal{T}_X$. Then, $\mathcal{T}_X$ admits a diagram $(\Sigma; \alpha, \beta, \gamma)$, and there exists a collection $\eta$ which consists of $l$ simple closed curves in $\Sigma$ such that:
\begin{enumerate}
\item[(1)] $(\Sigma; \alpha, \beta)$ is a standard diagram of $\Sigma_{p,b} \times I$;
\item[(2)] In $H_1(\Sigma)$, we have \[ [\gamma_i] = -[\alpha_i] - \sum_{j=1}^{g-p} \widetilde{Q}_{ji} [\beta_j] - d_i [\eta_i], \] where $Q$ is the intersection form of $X$ and $\widetilde{Q} = A_{\psi}^{-1} \oplus Q \oplus \langle 0 \rangle^{k-l}$.
\end{enumerate}
\end{theorem*}

\section{Several facts of relative trisection}
In this section, we see several facts related to relative trisection to prove the main result. 

We know an important fact showed by Waldhausen \cite{Wal1} that a Heegaard splitting of $S^3$ is equivalent to a standard splitting. Similarly, a splitting of connected sums of $S^1 \times S^2$ is also equivalent to a standard splitting by applying Haken's Lemma. Regarding sutured Heegaard splitting, we can show a similar uniqueness by using the following generalization of Haken's Lemma:

\begin{lemma}[Casson and Gordon \cite{CG1}]
Let $(C, C')$ be a sutured Heegaard splitting of $(M; B, B')$. Let $(S, \partial S) \subset (M, B \amalg B')$ be a disjoint union of essential 2-sphere and disks. Then, there exists a disjoint union of essential 2-spheres and disks $S^*$ in $M$ such that:
\begin{enumerate}
\item[(1)] $S^*$ is obtained from $S$ by ambient 1-surgery and isotopy;
\item[(2)] each component of $S^*$ meets $\Sigma$ in a single circle;
\item[(3)] there exist complete disk systems $D$, $D'$ for $C$, $C'$, respectively, such that $D \cap S^* = D' \cap S^* = \emptyset$.
\end{enumerate}
\end{lemma}

\noindent According to \cite{CG1}, if M is irreducible, and so $S$ must consist of disks only, then $S^*$ is isotopic to $S$.

\begin{proposition}
A sutured Heegaard splitting of $\Sigma_{p,b} \times I \, \sharp \, ( \sharp^{k-l} S^1 \times S^2)$ is equivalent to a standard splitting.
\end{proposition}
\begin{proof}
By applying a similar argument to the closed case, it is sufficient to show that a genus $g$ splitting of $\Sigma_{p,b} \times I$ is standard. 

Let $(C, C')$ be a genus $g$ sutured Heegaard splitting of $(\Sigma_{p,b} \times I; \Sigma_{p,b}, \Sigma_{p,b})$, and let $a$ be a collection of $l$ arcs which cuts $B$ into a disk. When $D$ is a regular neighborhood of $\partial a$ in $\partial B$, we can attach $l$ 2-dimensional 1-handles $^{2}h^1$ whose attaching regions are $D$ such that $^{2}h^1\times I$ are 3-dimensional 1-handles $h^1$ whose attaching regions are $D \times I$. Therefore, $(C \cup {}^2h^1 \times [0, 1/2], C' \cup {}^2h^1 \times [1/2, 1])$ is a sutured splitting of $((\Sigma_{p,b} \cup {}^2h^1) \times I; \Sigma_{p,b} \cup {}^2h^1, \Sigma_{p,b} \cup {}^2h^1)$. We call it $(\widetilde{C}, \widetilde{C}')$. 

Let $c_i$ be a core of the handle ${}^2h^1_i$. A curve $(a_i \cup c_i) \times I \cap \widetilde{B}'$ can be described as $a'_i \cup c_i$ for some arc $a'_i$ since $(a_i \cup c_i) \times I \cap h^1 = c_i \times I$. Since a collection $\{a_i \cup c_i \}$ is a cut system on $\widetilde{B}$, $\{a'_i \cup c_i \}$ is likewise on $\widetilde{B}'$. Attaching additional 3-dimensional 2-handles $h^2$ along it, we obtain a compression body $\widetilde{C}''$ from $\Sigma \cup {}^2h^1$ to $\widetilde{B}'_{a \cup c}$. Since $(h^1_i, h^2_i)$ is a canceling pair, $(\widetilde{C}, \widetilde{C}'')$ is a sutured splitting of $\Sigma_{p,b} \times I$.

 Let $D$ be a collection of disks obtained by capping $(a_i \cup c_i) \times I$ with a core disk of $h^2$. Applying the Lemma 3.1 for $D$, we obtain complete disk systems $\widetilde{D}$ and $\widetilde{D}''$ for $\widetilde{C}$ and $\widetilde{C}''$, respectively. $D_i \cap (\Sigma \cup {}^2h^1)$ is a single circle, and it splits into two arcs $c_i$ and $e_i$ in $\Sigma$ since $D \cap \widetilde{D} = \emptyset$. Therefore, $\widetilde{D}$ can be regarded as a complete disk system for $C$ satisfying $\partial \widetilde{D} \cap e = \emptyset$. A circle $e_i \cup c_i$ is capped with $D_i \cap \widetilde{C}''$. Since the collection of these disks is disjoint from $\widetilde{D}''$ and $\{e_i \cup c_i\}$ can be regarded as a cut system on $\Sigma$, we can choose another complete disk system for $\widetilde{C}''$ which has a subcollection $\{D_i \cap \widetilde{C}''\}$. When a disk in this system except $\{D_i \cap \widetilde{C}''\}$ intersects a co-core of ${}^2h^1_i$, we slide it over $D_i \cap \widetilde{C}''$ so that the disk can be disjoint from the co-core. Therefore, we can obtain a disk system $D'$ for $C'$ by excepting $\{D_i \cap \widetilde{C}''\}$ from the system for $\widetilde{C}''$ satisfying $\partial D' \cap e = \emptyset$.

Since $\partial D \cap e = \partial D' \cap e = \emptyset$, $\Sigma$ can be split into $\Sigma_{g-p}$ and $\Sigma_{p,b}$ such that the former includes $\partial D$, $\partial D'$ and the latter does $e$. It is clear that $(\Sigma_{g-p}; \partial D, \partial D')$ is a Heegaard diagram of $S^3$. By Waldhausen' theorem, $(\Sigma; \partial D, \partial D')$ can be standard.
\end{proof}

By this proposition, the definition of relative trisection in this paper is equivalent to what is given in \cite{CGPC1}. Therefore, the following corollary holds in regard to relative trisection diagram.

\begin{corollary}
Let $(\Sigma; \alpha^1, \alpha^2, \alpha^3)$ be a $(g, (k_1, k_2, k_3); p, b)$-relative trisection diagram. $(\Sigma; \alpha^i, \alpha^{i+1})$ is diffeomorphism and handle slide equivalent to a diagram $(\Sigma; \delta^{k_i}, \epsilon^{k_i})$, where $(\Sigma; \delta^k, \epsilon^k)$ is the following diagram. 
\begin{figure}[H]
\centering
\includegraphics[keepaspectratio,scale=0.8,pagebox=cropbox]{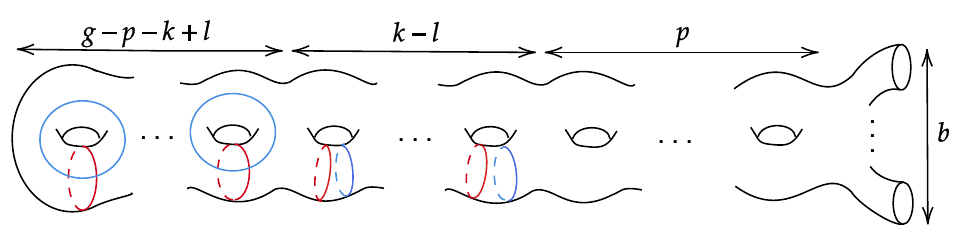}
\caption{A standard diagram $(\Sigma; \delta^{k}, \epsilon^{k})$, where the red curves are $\delta^k$, and blue curves are $\epsilon^k$.  }  %%%%Figure 1 of standard diagram
\end{figure}
\end{corollary}

The monodromy of the open book decomposition derived from a relative trisection can be described by the algorithm of Castro, Gay and Pinz\'on-Caicedo. For a cut system of curves $\alpha$ and a collection of arcs $a$ in $\Sigma$, we also call $(\alpha, a)$ a cut system if $a$ is disjoint from $\alpha$ and cuts $\Sigma_{\alpha}$ into a disk (when $\alpha = \emptyset$, we simply describe it as $a$).  A cut system $(\alpha, a)$ is handle slide and diffeomorphism equivalent to another $(\alpha', a')$ if $\alpha'$ is equivalent to $\alpha$ and $a'$ is obtained by sliding $a$ over $\alpha$.

\begin{proposition}[Theorem 5 in \cite{CGPC1}]
Let $a$ be a collection of $l$ arcs in $\Sigma$ such that $(\alpha^1, a)$ is a cut system. The monodromy $\psi : \Sigma_1 \to \Sigma_1$ induced by a relative trisection $\mathcal{T}_X$ which has a diagram $(\Sigma; \alpha^1, \alpha^2, \alpha^3)$ is described by the following algorithm. Putting $(\tilde{\alpha}^1, a^1) = (\alpha^1, a)$, we get cut systems $(\tilde{\alpha}^i, a^i)$ inductively. When there exists $(\tilde{\alpha}^i, a^i)$, we obtain $(\tilde{\alpha}^{i+1}, a^{i+1})$ such that $\tilde{\alpha}^{i+1}$ is equivalent to $\alpha^{i+1}$ and $(\tilde{\alpha}^i, a^{i+1})$ is equivalent to $(\tilde{\alpha}^i, a^i)$. Then, the pair $(\tilde{\alpha}^4, a^4)$ is equivalent to $(\alpha^1, a')$ for some collection of arcs $a'$. The monodromy $\psi$ is the unique (up to isotopy) map such that $\psi(a) = a'$.
\end{proposition}

Moussard and Schirmer consider a homological version of this algorithm in \cite{MS1}. We see this to prove the main theorem in this paper. Let $L_i$ be the subgroup of $H_1(\Sigma)$ generated by the homology classes of the curves $\alpha^i$. Since $(\Sigma; \alpha^i, \alpha^{i+1})$ can be standard, we can obtain collections $\kappa^i$ and $\lambda^{i+1}$ which consist of simple closed curves on $\Sigma$ that form bases of $\frac{L_i}{L_i \cap L_{i+1}}$ and $\frac{L_{i+1}}{L_i \cap L_{i+1}}$, respectively. For a collection of curves and proper arcs $\mu$ and a collection of curves $\nu$ on $\Sigma$, we define the matrix $\mu \cdot \nu$ whose element $(\mu \cdot \nu)_{ij}$ is $\langle [\mu_i], [\nu_j] \rangle_{\Sigma}$. We can define $\nu \cdot \mu$ similarly. 

\begin{proposition}
In the algorithm of Proposition 3.4, we can choose the collections of arcs $a^i$ satisfying the following formula;
\[ [a^{i+1}_j -a^i_j] = \sum R^i_{jn} [\kappa^i_n], \]
where $R^i = -(a^i \cdot \lambda^{i+1}) (\kappa^i \cdot \lambda^{i+1})^{-1}$.
\end{proposition}
\begin{proof}
Let $\mu^i$ be a collection of curves on $\Sigma$ such that $\{[\kappa^i_j], [\mu^i_m]\}$ and $\{[\lambda^{i+1}_i], [\mu^i_m]\}$ form a basis of $L_i$ and $L_{i+1}$, respectively. From the algorithm to obtain $a^{i+1}$, we can represent $[a^{i+1}_j - a^i_j]$ in $H_1(\Sigma, \partial \Sigma)$ by using matrices $S^i$ and $T^i$ as follows;
\[ [a^{i+1}_j -a^i_j] = \sum S^i_{jk} [\kappa^i_k] + \sum T^i_{jm} [\mu^i_m]. \]
Since $a^{i+1}$ and $\mu^i$ are disjoint from $\tilde{\alpha}^{i+1}$, and so $\lambda^{i+1}$, by multiplying them from the right, we obtain $- (a^i \cdot \lambda^{i+1}) = S^i (\kappa^i \cdot \lambda^{i+1})$. Therefore, we have $S^i = R^i$. If there exist nonzero components of $T^i$, we slide $a^{i+1}$ over $\mu^i$ several times. This sliding can be regarded as sliding of $a^i$ over $\tilde{\alpha}^i$ since $\mu^i$ is obtained by sliding of $\tilde{\alpha}^i$. Moreover, since $\mu^i$ is also obtained by sliding of $\tilde{\alpha}^{i+1}$, the sliding over $\mu^i$ is equivalent to isotopies of arcs in $\Sigma_{i+1}$, so the disjointness between $a^{i+1}$ and $\tilde{\alpha}^{i+1}$ is kept. Therefore, we can choose $a^{i+1}$ satisfying $T^i = 0$.
\end{proof}

We also review the calculation of the homology by using trisection diagram. Let $L^{\partial}_i$ be a subgroup of $H_1(\Sigma, \partial \Sigma)$ generated by the homology classes of curves and arcs in a cut system $(\alpha^i, a^i)$ for $\Sigma$.

\begin{proposition}[Theorem 1 in \cite{Tan1}]
The homology of $X$ can be obtained from the following chain complex: \vspace{-1ex}
{\small \begin{figure}[H]
\begin{tikzcd}
            0 \ar{r} &[-1em] (L_1 \cap L_3) \oplus (L_2 \cap L_3) \ar{r}{\pi} &[-0.25em] L_3 \ar{r}{\rho} &[-0.25em] \Hom(L_1^{\partial} \cap L_2^{\partial} , \mathbb{Z}) \ar{r}{0} &[-0.5em] \mathbb{Z} \ar{r} &[-1em] 0,
\end{tikzcd}
\end{figure}} \vspace{-2ex}
\noindent
where $\pi (x,y) = x+y,  \rho(x)(y) = \langle y , x \rangle_{\Sigma}$ and the $0$ at the left end is the 4-chain complex.
\end{proposition}

Moreover, the linking matrix regarding to a special handle decomposition derived from the relative trisection can be calculated by the diagram. The {\it linking matrix of $\mathcal{T}_X$ regarding $\gamma$} is the matrix whose components are linking numbers $\lk(\gamma_i, \gamma_j)$ in $\partial X_1$ by regarding it as submanifold of $S^3$.

\begin{proposition}[Proposition 6.3 in \cite{Tan1}]
Suppose $(\Sigma; \alpha, \beta)$ is a standard diagram of $\Sigma_{p,b} \times I \, \sharp \, ( \sharp^{k_1-l} S^1 \times S^2)$. The linking matrix of $\mathcal{T}_X$ regarding $\gamma$ is $(\gamma \cdot (\beta, a)) \, R^g_{p,b} \, ((\alpha, a) \cdot \gamma)$, where $R^g_{p,b}$ is $ I_{g-p} \oplus^p \begin{pmatrix} 0 & 0 \\ 1 & 0 \end{pmatrix} \oplus O_{b-1}$.
\end{proposition}

\section{A property of $\gamma$}
In this section, we give a proof of the main theorem. It is shown in the same way as the first part in the proof of the Theorem 4.4 in \cite{FKSZ1} that a $(g, (l, l, k); p, b)$-relative trisection of $X$ has a similar diagram to the Theorem without additional conditions.

\begin{lemma}
A  $(g, (l, l, k); p, b)$-relative trisection $\mathcal{T}_X$ has a diagram $(\Sigma; \alpha, \beta, \gamma)$ such that:
\begin{enumerate}
\item[(1)] $(\Sigma; \alpha, \beta)$ is a standard diagram of $\Sigma_{p,b} \times I$;
\item[(2)] $\gamma \cdot \beta = \beta \cdot \alpha = I_{g-p}$, where $I_{g-p}$ is the $g-p$ identity matrix.
\end{enumerate}
\end{lemma}

Even though we suppose $X$ has a $(g, (l, l, k); p, b)$-relative trisection, $X$ may not be simply connected. We impose $H_1(X) = 0$ to get in the similar situation to the closed case.

\begin{lemma}
Suppose $H_1(X) = 0$, and $X$ has a $(g, (l, l, k); p, b)$-relative trisection $\mathcal{T}_X$. Then, $\mathcal{T}_X$ admits a diagram $(\Sigma; \alpha, \beta, \gamma)$ and there exists a collection $\eta$ which consists of $l$ simple closed curves in $\Sigma$ such that:
\begin{enumerate}
\item[(1)] $(\Sigma; \alpha, \beta)$ is a standard diagram of $\Sigma_{p,b} \times I$;
\item[(2)] In $H_1(\Sigma)$, we have  \[ [\gamma_i] = -[\alpha_i] - \sum_{j=1}^{g-p} (\alpha \cdot \gamma)_{ji} [\beta_j] - d_i [\eta_i]; \]
\item[(3)] $(\alpha \cdot \gamma)_{ji} = (\alpha \cdot \gamma)_{ij}$ for $i > l$, and $(\alpha \cdot \gamma)_{ji} = 0$ for $i > l+b_2(X)$.
\end{enumerate}
\end{lemma}

\begin{proof}
At first, we take a diagram $(\Sigma; \alpha, \beta, \gamma)$ of $\mathcal{T}_X$ by applying Lemma 4.1. From the standardness of $(\Sigma; \alpha, \beta)$, we can obtain a collection of arcs $a$ on $\Sigma$ which is disjoint from $\alpha$ and $\beta$. Additionally, we take a collection of curves $\eta$ on $\Sigma_{\alpha}$ as the following diagram.
\begin{figure}[H]
\centering
\includegraphics[keepaspectratio,scale=0.9,pagebox=cropbox]{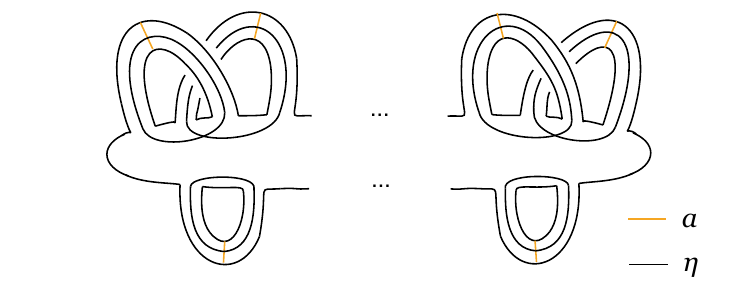}
\caption{The curves $a$ and $\eta$.}
\end{figure} \noindent

\noindent Since $a$ is disjoint from $\alpha$ and $\beta$, $\eta$ can be regarded as curves on $\Sigma$. We orient $a$ and $\eta$ satisfying $\langle [a_i], [\eta_j] \rangle_{\Sigma} = \delta_{ij}$. Since $\{[\alpha_i], [a_j]\}$ and $\{[\beta_i], [a_j]\}$ form a basis of $L_1^{\partial}$ and $L_2^{\partial}$, respectively. Therefore, $\{[a_j]\}$ forms a basis of $L_1^{\partial} \cap L_2^{\partial}$, and $\{[\eta_j]\}$ can be regarded as a basis of $\Hom(L_1^{\partial} \cap L_2^{\partial}, \mathbb{Z})$.

From the assumption $H_1(X) = 0$, the 2-boundary map $\rho$ is surjective. Moreover, since $\Hom(L_1^{\partial} \cap L_2^{\partial}, \mathbb{Z})$ is free, the following exact sequence split: \vspace{-1ex}
\begin{figure}[H]
\begin{tikzcd}
            0 \ar{r} & \ker \rho \ar{r} & L_3 \ar{r}{\rho} & \Hom(L_1^{\partial} \cap L_2^{\partial} , \mathbb{Z}) \ar{r} & 0.
\end{tikzcd}
\end{figure}
\vspace{-1ex}
\noindent We take a basis of $L_{\gamma}$ such that each first $l$ component corresponds to $-[\eta_i]$ by $\rho$. Because a change of basis can be represented by a sequence of handle slides, this basis is represented by another cut system $\gamma$ for $C_3$. For a unimodular matrix, we perform a sequence of handle slides regarding $\beta$ such that its change of basis is represented by the given matrix. Then, performing a sequence of handle slides regarding $\alpha$, we can keep a standardness of $(\Sigma; \alpha, \beta)$ (the same argument as the proof of the Theorem 4.4 in \cite{FKSZ1}). Note that these slides are performed to be disjoint from $a$ since $\Sigma$ split into $\Sigma'$ and $P$ such that $\alpha$ and $\beta$ are in $\Sigma'$ and $a$ is in $P$. Therefore, when we obtain the new cut system $\gamma$, we can keep the conditions (1) and (2) of the Lemma 4.1. Moreover, since $\{[\alpha_i], [\beta_i], [\eta_j]\}$ is an almost symplectic basis of $H_1(\Sigma)$, we have the formula of the condition (2).

$H_2(X)$ has no torsion since $H_1(X) = 0$. Then, the following exact sequence also split: \vspace{-1ex}
\begin{figure}[H]
\begin{tikzcd}
            0 \ar{r} & L_1 \cap L_3 \ar{r} & \ker \rho \ar{r} & H_2(X) \ar{r} & 0. 
\end{tikzcd}
\end{figure}
\vspace{-1ex}
\noindent $L_3$ is isomorphic to $\langle [\gamma_1], \dots , [\gamma_l] \rangle \oplus H_2(X) \oplus L_1 \cap L_3$. Applying the same argument as before to a basis of $\ker \rho$ regarding this correspondence, we also obtain the new $\gamma$. This $\gamma$ still satisfies the condition (2). Since $[\gamma_{i+l+b_2(X)}]$ are elements of $L_1 \cap L_3$, its coefficients of $[\beta_j]$ must be $0$. Therefore, we have $(\alpha \cdot \gamma)_{ji} = 0$ for $i > l + b_2(X)$.

Calculating the linking matrix of $\mathcal{T}_X$ regarding this $\gamma$ by using Proposition 3.7, we have
\[(\gamma \cdot (\beta, a)) \, R^g_{p,b} \, ((\alpha, a) \cdot \gamma) = \alpha \cdot \gamma - \oplus^p \begin{pmatrix} 0 & 0 \\ 1 & 0 \end{pmatrix} \oplus O_{g-3p}. \]
For $i > l$, $(\alpha \cdot \gamma)_{ji}$ is equal to the $(j,i)$-component of the linking matrix. Since linking matrix is symmetric, the condition (3) holds.
\end{proof}

Before starting the proof of the main theorem, we prepare an algebraic lemma in regard to a bilinear form $(A, Q)$. This lemma is similar to the Lemma 1.2.12 in \cite{GS1}.

\begin{lemma}
Let $Q$ be a bilinear form on a free group $A = A_1 \oplus A_2$ satisfying $Q(a, a_1) = Q(a_1, a)$ for all $a \in A$ and $a_1 \in A_1$. Suppose that the restriction of $Q$ to $A_1$ is unimodular. Them, $(A, Q)$ can be split as the sum of forms $(A_1, Q|A_1) \oplus (A_1^{\bot}, Q|A_1^{\bot})$, where $A_1^{\bot} = \{ y \in A \, \, | \, \, Q(x, y) = 0 \, \, \text{for all} \, \, x \in A_1\}.$ Moreover,  an isomorphism $\phi : A_2 \to A_1^{\bot}$ can be obtained such that there exists $b \in A_1$ for $x \in A_2$ and $\phi(x) = x-b$. 
\end{lemma}
\begin{proof}
$A_1 \cap A_1^{\bot} = 0$ is clear from the unimodularity of $Q|A_1$. For any $x \in A$, we can take the linear function $a \mapsto Q(a, x)$ on $A_1$. By the unimodularity of $Q|A_1$, there is a unique element $b_x \in A_1$ satisfying $Q(a, x) = Q(a, b_x)$ for all $a \in A_1$. Therefore, $x-b_x \in A_1^{\bot}$, and then we can see $A = A_1 \oplus A_1^{\bot}$. Moreover, $Q = Q|A_1 \oplus Q|A_1^{\bot}$ can be proven by using the symmetry that $Q(a, a_1)$ is equal to $Q(a_1, a)$ for all $a \in A$ and $a_1 \in A_1$.

Let $f : A \to A_1^{\bot}$ be the homomorphism defined by $f(x) = x-b_x$. $\ker f = A_1$ is clear. For $c \in A_1^{\bot}$, there exist $a_i \in A_i$ such that $c = a_1 + a_2$. Since $b_{a_2} = -a_1$, we have $f (a_2) = c$. Therefore, we should define $\phi$ as the restriction of $f$ to $A_2$.
\end{proof}

We give a proof of the main theorem here.

\begin{theorem*}
Suppose $H_1(X) = H_1(\partial X) = 0$, and $X$ has a $(g, (l, l, k); p, b)$-relative trisection $\mathcal{T}_X$. Then, $\mathcal{T}_X$ admits a diagram $(\Sigma; \alpha, \beta, \gamma)$, and there exists a collection $\eta$ which consists of $l$ simple closed curves in $\Sigma$ such that:
\begin{enumerate}
\item[(1)] $(\Sigma; \alpha, \beta)$ is a standard diagram of $\Sigma_{p,b} \times I$;
\item[(2)] In $H_1(\Sigma)$, we have \[ [\gamma_i] = -[\alpha_i] - \sum_{j=1}^{g-p} \widetilde{Q}_{ji} [\beta_j] - d_i [\eta_i], \] where $Q$ is the intersection form of $X$ and $\widetilde{Q} = A_{\psi}^{-1} \oplus Q \oplus \langle 0 \rangle^{k-l}$.
\end{enumerate}
\end{theorem*}
\begin{proof}
Firstly, we take a diagram $(\Sigma; \alpha, \beta, \gamma)$ satisfying the conditions of Lemma 4.2. The matrix $\alpha \cdot \gamma$ can be expressed as follows:
\[ \alpha \cdot \gamma = \begin{pmatrix} C & D \\ {}^tD & Q \end{pmatrix} \oplus \langle 0 \rangle^{k-l}, \]
where $C$, $D$ and $Q$ are $(l, l)$-, $(l, b_2)$- and $(b_2, b_2)$-matrices, respectively. Since $\{ [\gamma_{l+1}], \dots , [\gamma_{l+b_2}] \}$ forms a basis of $H_2(X)$ and $\alpha \cdot \gamma$ is equal to the linking matrix for this basis, $Q$ is the representation matrix of the intersection form of $X$. From the assumption $H_1(\partial X) = 0$, $Q$ is unimodular. Applying the Lemma 4.3 to the first summand of $\alpha \cdot \gamma$ such that $A_1 = \langle [\gamma_{l+1}], \dots , [\gamma_{l+b_2}] \rangle$ and $A_2 = \langle [\gamma_1], \dots , [\gamma_l] \rangle$, we take the new basis $\{\phi([\gamma_1]), \dots , \phi([\gamma_l]), [\gamma_{l+1}], \dots , [\gamma_{l+b_2}]\}$ of $\langle [\gamma_1], \dots , [\gamma_{l+b_2}] \rangle$. In the same way as before, we obtain the new $\gamma$ regarding this basis. Since $\phi([\gamma_i]) = [\gamma_i] - b_i$ for $b_i \in A_1$, the new diagram $(\Sigma; \alpha, \beta, \gamma)$ still satisfies the conditions of the Lemma 4.2. Moreover, $\alpha \cdot \gamma$ can be expressed as $B \oplus Q \oplus \langle 0 \rangle^{k-l}$, where $B$ is a $(l, l)$-matrix.

We should prove $B$ is the inverse of the monodromy representation matrix. Note that $B$ is invertible since $\kappa^3 \cdot \lambda^4 = B \oplus Q$ is unimodular. Applying the Proposition 3.5 to $(\Sigma; \alpha, \beta, \gamma)$, we have the following formulas in $H_1(\Sigma)$:
\begin{align}
[a^2_i - a^1_i] &= 0, \notag \\ 
[a^3_i - a^2_i] &= - [\beta_i], \notag \\ 
[a^4_i - a_i] &= \sum^l_{j=1} B^{-1}_{ji}([\alpha_j] + [\eta_j]). \notag
\end{align}
Therefore, $B^{-1}$ is the representation matrix of $\xi_{\psi} : H_1(\Sigma_1, \partial \Sigma_1) \to H_1(\Sigma_1)$ regarding bases $\{[a_i]\}$ and $\{[\eta_i]\}$.
\end{proof}

\section{Torelli group of $\Sigma$}
We see a relation between $(g, (l, l, k); p, b)$-trisection diagram and the Torelli group of $\Sigma$ in this section. We define the Torelli group $\mathcal{I}(\Sigma)$ following \cite{Chu1}. 

Let $\widehat{\Sigma}$ be a surface obtained by gluing $\Sigma_{0, b+1}$ to $\Sigma$ along $b$ boundary components of $\Sigma_{0,b+1}$ and $\partial \Sigma$. Note that $\widehat{\Sigma}$ is the totally separated surface constructed by the nonseparating partition of $\Sigma$ in \cite{Chu1}. We define $H(\Sigma)$ as $H_1(\widehat{\Sigma})$.

Let $\mathcal{M}(\Sigma)$ be the mapping class group of $\Sigma$, that is, the group of the rel $\partial \Sigma$ isotopy class of the self-homomorphism fixing boundaries. The natural inclusion $i : \Sigma \to \widehat{\Sigma}$ induces a homomorphism $i_* : \mathcal{M}(\Sigma) \to \mathcal{M}(\widehat{\Sigma})$ by using the identity on $\Sigma_{0,b+1}$. Composing it with the natural action of $\mathcal{M}(\widehat{\Sigma})$ on $H(\Sigma)$, we obtain an action of $\mathcal{M}(\Sigma)$ on $H(\Sigma)$. The {\it Torelli group of $\Sigma$} is defined by using this action as follows:
\[ \mathcal{I}(\Sigma) := \{\phi \in \mathcal{M}(\Sigma) \, \, | \, \, \phi \, \, \text{acts trivially on} \, \, H(\Sigma)\}. \]

Since any $\phi \in \mathcal{M}(\Sigma)$ has the action $\xi_{\phi} : H_1(\Sigma, \partial \Sigma) \to H_1(\Sigma)$, we can define another subgroup as follows:
\[ \mathcal{I}'(\Sigma) := \{\phi \in \mathcal{M}(\Sigma) \, \, | \, \, \xi_{\phi}(x) = 0 \, \, \text{for all} \, \, x \in H_1(\Sigma, \partial \Sigma) \}. \]

\begin{lemma}
$\mathcal{I}(\Sigma) = \mathcal{I}'(\Sigma)$.
\end{lemma}
\begin{proof}
Let $(\Sigma; \alpha, \beta)$ be a standard diagram of $\Sigma_{0, b} \times I$. We can take a collection $a$ of $b-1$ arcs disjoint from $\alpha$, $\beta$ such that $\partial^- a_i \in b_i$ and $\partial^+ a_i \in b_{i+1}$, where $\{b_i\}$ is boundary components of $\Sigma$. The collections $\{[\alpha_i], [\beta_i], [a_j]\}$ and $\{[\alpha_i], [\beta_i], [b_j]\}_{1 \le j \le b-1}$ form bases of $H_1(\Sigma, \partial \Sigma)$ and $H_1(\Sigma)$, respectively. We take a collection $c$ of $b-1$ arcs on $\Sigma_{0, b+1}$ such that $\partial c = \partial a$ in $\partial \Sigma$. The collection $\{ [\alpha_i], [\beta_i], [a_j - c_j], [b_j]\}_{1 \le j \le b-1}$ forms a symplectic basis of $H(\Sigma)$. The homomorphism $i_* : H_1(\Sigma) \to H(\Sigma)$ induced by the natural inclusion $i : \Sigma \to \widehat{\Sigma}$ is injective.

We have $\phi_* ([a_j -c_j]) = [\phi(a_j)-c_j]$, $[\phi(a_j) -c_j] - [a_j - c_j] = i_*([\phi(a_j)-a_j])$ and $\phi_*([b_j]) = [b_j]$ in $H(\Sigma)$ for $\phi \in \mathcal{M}(\Sigma)$. For $\phi \in \mathcal{I}'(\Sigma)$, we get $\phi_* ([a_j - c_j]) - [a_j - c_j] = 0$ in $H(\Sigma)$, and then $\phi \in \mathcal{I}(\Sigma)$. Conversely, we get $i_*([\phi(a_j) - a_j]) = 0$ for $\phi \in \mathcal{I}(\Sigma)$, and then $\phi \in \mathcal{I}'(\Sigma)$.
\end{proof}

We can obtain the following corollary corresponding to the Corollary 3.12 in \cite{L-C1}.

\begin{corollary}
Let $X$, $Y$ be smooth, oriented 4-manifolds with connected boundaries such that $Q_X \cong Q_Y$. Suppose that $H_1(X) = H_1(\partial X) = H_1(Y) = H_1(\partial Y) = 0$ and they admit $(g, (l, l, k); p, b)$-trisection $\mathcal{T}_X$, $\mathcal{T}_Y$ such that the actions of induced monodromies are the same. Then, there exists a diagram $(\Sigma; \alpha, \beta, \gamma)$ for $\mathcal{T}_X$ and an element $\rho \in \mathcal{I}(\Sigma)$ such that $(\Sigma; \alpha, \beta, \rho(\gamma))$ is a diagram for $\mathcal{T}_Y$.
\end{corollary}
\begin{proof}
Applying the Theorem to $\mathcal{T}_X$ and $\mathcal{T}_Y$ for the same $(\Sigma; \alpha, \beta)$, $a$ and $\eta$, we can obtain diagrams $(\Sigma; \alpha, \beta, \gamma)$ and $(\Sigma; \alpha, \beta, \hat{\gamma})$ for $\mathcal{T}_X$ and $\mathcal{T}_Y$, respectively. Moreover, we get collections $a^3$, $\hat{a}^3$ of arcs and colllections $\widetilde{\gamma}$, $\widetilde{\hat{\gamma}}$ of curves satisfying $ [\hat{a}^3_i - a^3_i] = 0$ in $H_1(\Sigma)$ and $\widetilde{\gamma} \cap a^3 = \widetilde{\hat{\gamma}} \cap \hat{a}^3 = \emptyset$ in the algorithm to calculate the monodromy. Let $P$ be the matrix such that $[\widetilde{\gamma}_i] = \sum P_{ji} [\gamma_j]$. Since $(\widetilde{\gamma} \cup \beta ) \cap a^3 = \emptyset$, $\Sigma$ split into surfaces $\Sigma'$ and $P$ such that $(\widetilde{\gamma} \cup \beta ) \subset \Sigma'$, $a^3 \subset P$. We perform handle slides regarding $\widetilde{\gamma}$ in $\Sigma'$ to get a new $\gamma'$ satisfying $[\gamma'_i] = \sum P^{-1}_{ji} [\widetilde{\gamma}_j]$. Note that we have $[\gamma'_i] = [\gamma_i]$ and $\gamma' \cap a^3 =\emptyset$. 

Applying the same argument to $(\Sigma; \alpha, \beta, \hat{\gamma})$, we obtain $\hat{\gamma}'$. Diagrams $(\Sigma; \alpha, \beta, \gamma')$ and $(\Sigma; \alpha, \beta, \hat{\gamma}')$ satisfy the conditions of the Theorem. In the same way as the Proposition 2.1 in \cite{L-C1}, by using $a^3$ and $\hat{a}^3$ additionally, $\gamma'$ and $\hat{\gamma}'$ can be extended to geometric symplectic bases $\{\gamma', g'\}$ and $\{\hat{\gamma}', \hat{g}'\}$ in $\Sigma'$ and $\widehat{\Sigma}'$, respectively, such that $[g'_i] = [\hat{g}'_i]$ in $H_1(\Sigma)$. We can take a homeomorphism $\rho$ of $\Sigma$ which sends $\gamma'_i$ to $\hat{\gamma}'_i$, $g'_i$ to $\hat{g}'_i$ and $a^3_j$ to $\hat{a}^3_j$. It is clear that $\xi_{\rho}(x) = 0$ for all $x \in H_1(\Sigma, \partial \Sigma)$, and then $\rho \in \mathcal{I}(\Sigma)$.
\end{proof}

We relate this corollary to exotic 4-manifolds. We call a connected, oriented compact 4-manifold $X$ {\it 2-handlebody} if $X$ has a handle decomposition without 3-, 4-handles. Any 2-handlebody $X$ admits a $(g, l; p,b)$-relative trisection from the following propositions:

\begin{proposition}[Theorem 2.1 in \cite{Har1}]
Any 2-handlebody $X$ admits an achiral Lefshcetz fibration over $D^2$ with bounded fibers.
\end{proposition}

\begin{proposition}[Corollary 18 in \cite{CGPC1}]
Let $\pi : X \to D^2$ be an achiral Lefschetz fibration with regular fiber $\Sigma_{p,b}$. The manifold $X$ admits a $(g, l; p, b)$-trisection.
\end{proposition}

Let $X$, $Y$ be 2-handlebodies whose boundaries are diffeomorphic to the same integral homology 3-sphere. Any two open book decompositions of an integral homology 3-sphere are isotopic after several Hopf stabilizations \cite{GG1}. According to \cite{Cas1}, we can perform the relative stabilization corresponding to the Hopf stabilization of the induced open book decomposition. The integers $l$ and $k_i$ increase by $1$ after this stabilization. Therefore, $X$ and $Y$ admit $(g, l; p, b)$-trisection $\mathcal{T}_X$ and $(g', l; p, b)$-trisection $\mathcal{T}_Y$, respectively, such that they induce the same open book decomposition of boundary. We can obtain the following corollary regarding 2-handlebody.

\begin{corollary}
Let $X$, $Y$ be 2-handlebodies satisfying $H_1(X) = H_1(\partial X) = H_1(Y) = H_1(\partial Y) = 0$. If $X$ is homeomorphic to $Y$, there are $(g, l; p, b)$-trisection diagrams $(\Sigma; \alpha, \beta, \gamma)$ of $X$ and  an element $\rho \in \mathcal{I}(\Sigma)$ such that $(\Sigma; \alpha, \beta, \rho(\gamma))$ is a diagram of $Y$.
\end{corollary}
\begin{proof}
From the assumptions, there are $(g, l; p, b)$-trisection $\mathcal{T}_X$ and $(g', l; p, b)$-trisection $\mathcal{T}_Y$ which induce the same open book decomposition. Applying the theorem to $\mathcal{T}_X$, we obtain a diagram $(\Sigma; \alpha, \beta, \gamma)$. Since $k=l$, we have $g-p = l+b_2$. Therefore, $g$ is determined by $(p,b)$ and $b_2$, and then $g'$ must be equal to $g$. We should use the Corollary 5.2 to finish the proof.
\end{proof}

Using this Corollary, we can obtain the candidates for exotic manifolds to a 2-handlebody $X$. Moreover, we can apply this to geometrically simply connected closed 4-manifold. 

\begin{corollary}
Let $X$, $Y$ be oriented, geometrically simply connected, closed 4-manifolds. If $X$ is homeomorphic to $Y$, there are $(g, l; p, b)$-trisection diagram $(\Sigma; \alpha, \beta, \gamma)$ of $X \setminus D^4$ and an element $\rho \in \mathcal{I}(\Sigma)$ such that $(\Sigma; \alpha, \beta, \rho(\gamma))$ is a diagram of $Y \setminus D^4$.
\end{corollary}
\begin{proof}
$X \setminus D^4$, $Y \setminus D^4$ can be homeomorphic 2-handlebodies. It is clear that they satisfy the other assumptions of the Corollary 5.5.
\end{proof}

\appendix

\section{Johnson kernel of $\Sigma$}
For a surface $S$ with one boundary component,  Johnson defined the Johnson homomorphism $\tau_S : \mathcal{I}(S) \to (H_1(S))^* \otimes \bigwedge^2 H_1(S)$ and proved that the image of this map is $\bigwedge^3 H_1(S)$ ($\subset  (H_1(S))^* \otimes \bigwedge^2 H_1(S)$) in \cite{Joh1}. Moreover, Johnson also proved in \cite{Joh2} that the Johnson kernel $\mathcal{K}(S)$, or the kernel of $\tau_S$, is generated by Dehn twists on separating curves.

For $\Sigma \hookrightarrow \widehat{\Sigma}$, restricting $\tau_{\widehat{\Sigma}}$ to $\mathcal{I}(\Sigma)$, we can obtain the Johnson homomorphism $\tau_{\Sigma}$ \cite{Put1}. This is equivalent to the definition obtained by using non-separating partition of $\Sigma$ in \cite{Chu1}. It is followed from the Theorem 5.9 in \cite{Chu1} that the image of $\tau_{\Sigma}$ is $\bigwedge^3 H_1(\Sigma)$. 

For this $\tau_{\Sigma}$, the Johnson kernel $\mathcal{K}(\Sigma)$ is equivalent to $\mathcal{K}(\Sigma, \widehat{\Sigma})$ of \cite{Put1}. Note that a curve $c$ in $\Sigma$ is separating if $[c] = 0$ in $H_1(\Sigma)$. According to the Theorem 5.1 in \cite{Put1}, if the genus of $\Sigma$ is at least $2$, the group $\mathcal{K}(\Sigma)$ is generated by Dehn twists on separating curves.

Before applying the Johnson kernel to relative trisection diagram, we see other versions of the theorem in $(l, k, l)$ and $(k, l, l)$ cases.
\begin{theorem*}
Suppose $H_1(X) = H_1(\partial X) = 0$, and $X$ has a $(g, (l, k, l); p, b)$-relative trisection $\mathcal{T}_X$. Then, $\mathcal{T}_X$ admits a diagram $(\Sigma; \alpha, \beta, \gamma)$, and there exists a collection $\eta$ which consists of $l$ simple closed curves in $\Sigma$ such that:
\begin{enumerate}
\item[(1)] $(\Sigma; \alpha, \beta)$ is a standard diagram of $\Sigma_{p,b} \times I$;
\item[(2)] In $H_1(\Sigma)$, we have \[ [\gamma_i] = - \sum_{j=1}^{g-p} \widetilde{Q}_{ij} [\alpha_j] - [\beta_i] - d_i [\eta_i], \] where $Q$ is the intersection form of $X$ and $\widetilde{Q} = A_{\psi}^{-1} \oplus Q \oplus \langle 0 \rangle^{k-l}$.
\end{enumerate}
\end{theorem*}

Let $\{ \zeta_{g-p-k+l+1}, \dots , \zeta_{g-p} \}$ be a collection of curves on $\Sigma$ such that $\{ \alpha_i, \beta_i, \\ \alpha_{g-p-k+l+j}, \zeta_{g-p-k+l+j} \}$ is geometrically symplectic for a standard diagram $(\Sigma ; \alpha, \beta)$ of $\Sigma_{p,b} \times I \, \sharp \, ( \sharp^{k-l} S^1 \times S^2)$.
\begin{theorem*}
Suppose $H_1(X) = H_1(\partial X) = 0$, and $X$ has a $(g, (k, l, l); p, b)$-relative trisection $\mathcal{T}_X$. Then, $\mathcal{T}_X$ admits a diagram $(\Sigma; \alpha, \beta, \gamma)$, and there exists a collection $\eta$ which consists of $l$ simple closed curves in $\Sigma$ such that:
\begin{enumerate}
\item[(1)] $(\Sigma; \alpha, \beta)$ is a standard diagram of $\Sigma_{p,b} \times I \, \sharp \, ( \sharp^{k-l} S^1 \times S^2)$, and then there exists a collection $\zeta$ for this diagram;
\item[(2)] In $H_1(\Sigma)$, we have \[ [\gamma_i] = \begin{cases} -[\alpha_i] - \sum_{j=1}^{g-p} \widetilde{Q}_{ji} [\beta_j] - d_i [\eta_i] & (i \le g-p-k+l), \\ - [\zeta_i] & ( i > g-p-k+l),  \end{cases} \] where $Q$ is the intersection form of $X$ and $\widetilde{Q} = A_{\psi}^{-1} \oplus Q \oplus \langle 0 \rangle^{k-l}$.
\end{enumerate}
\end{theorem*}
\noindent For these two versions of the theorem, the same statement as the Corollary 5.2 also holds.

In the closed case, the Theorem 1.1 in \cite{L-C1} says that an element $\rho$ of the Torelli group can be reduced by an element of the Johnson kernel in the similar situation to the Corollary 5.2 of this paper (regarding the last $(k, l, l)$-version of the theorem). Therefore, a natural question arises:

\begin{question*}
In the situation of the Corollary 5.2, can $\rho \in \mathcal{I}(\Sigma)$ be reduced by an element of $\mathcal{K}(\Sigma)$?
\end{question*}

\section{Another boundary case}
We can prove that the similar situation holds in the case where $\partial X$ is diffeomorphic to $\sharp^r S^1 \times S^2$. However, we need to impose more conditions to the monodromy of $\partial X$, which it might be avoided essentially.

\begin{theorem*}
Suppose $H_1(X) = 0$, $\partial X \cong \sharp^r S^1 \times S^2$, $X$ has a $(g, (l, l, k); p, b)$-relative trisection $\mathcal{T}_X$ satisfying $b-1 \ge r$, and there exist a cut system of arcs $a$ for $\Sigma_1$ and a collection of curves $\eta$ satisfying $\langle [a_i], [\eta_j] \rangle_{\Sigma_1} = \delta_{ij}$ such that the representation matrix of the monodromy action $A_{\psi}$ regarding bases $\{[a_i]\}$ and $\{[\eta_i]\}$ is $A \oplus \langle 0 \rangle^r$, where $A$ is a unimodular $(l-r, l-r)$-matrix, and $\{[\eta_{2k-1}], [\eta_{2k}]\}_{k \le p}$ is symplectic. Then, $\mathcal{T}_X$ admits a diagram $(\Sigma; \alpha, \beta, \gamma)$ such that:
\begin{enumerate}
\item[(1)] $(\Sigma; \alpha, \beta)$ is a standard diagram of $\Sigma_{p,b} \times I$;
\item[(2)] In $H_1(\Sigma)$, we have \[ [\gamma_i] = -[\alpha_i] - \sum_{j=1}^{g-p} \widetilde{Q}_{ji} [\beta_j] - d_i [\eta_i], \] where $\widetilde{Q} = A^{-1} \oplus \begin{pmatrix} O & I_r \\ I_r & O \end{pmatrix} \oplus Q \oplus \langle 0 \rangle^{k-l}$ and $Q$ is a unimodular $(b_2-r, b_2-r)$-matrix such that $\langle 0 \rangle^r \oplus Q$ is the representation matrix of the intersection form of $X$.
\end{enumerate}
\end{theorem*}

To prove this version, we review the several calculations of the homology and the intersection form by using trisection diagram.

\begin{proposition}[Theorem 2 in \cite{Tan1}]
The homology of $X$ can also be obtained from the following chain complex $C^Z$: \vspace{-1ex}
{\small \begin{figure}[H]
\begin{tikzcd}
  0 \ar{r} &[-1.5em] (L_{1} \cap L_{2}) \oplus (L_{2} \cap L_{3}) \oplus (L_{3} \cap L_{1}) \ar{r}{\zeta} &[-1.2em] L_{1} \oplus L_2 \oplus L_{3} \ar{r}{\iota} &[-1.2em] H_1(\Sigma) \ar{r}{0} &[-1.2em] \mathbb{Z} \ar{r} &[-1.5em] 0,
\end{tikzcd}
\end{figure}} \vspace{-2ex}
\noindent
where $\zeta (x, y, z) = (x-z, y-x, z-y)$ and $\iota$ is a homomorphsim induced by the inclusions $\iota_{\nu}$.
\end{proposition}

\begin{proposition}[Theorem 4.9 in \cite{MS1}]
The homology of $(X, \partial X)$  can be obtained from the following chain complex $C^Z_{\partial}$: \vspace{-1ex}
{\small \begin{figure}[H]
\begin{tikzcd}
  \mathbb{Z} \ar{r}{0} &[-1.5em] (L^{\partial}_{1} \cap L^{\partial}_{2}) \oplus (L^{\partial}_{2} \cap L^{\partial}_{3}) \oplus (L^{\partial}_{3} \cap L^{\partial}_{1}) \ar{r}{\zeta^{\partial}} &[-1.2em] L^{\partial}_{1} \oplus L^{\partial}_2 \oplus L^{\partial}_{3} \ar{r}{\iota^{\partial}} &[-1.2em] H_1(\Sigma, \partial \Sigma) \ar{r} &[-1.2em] 0,
\end{tikzcd}
\end{figure}} \vspace{-2ex}
\noindent
where $\zeta^{\partial} (x, y, z) = (x-z, y-x, z-y)$ and $\iota^{\partial}$ is a homomorphsim induced by the inclusions $\iota_i^{\partial}$.
\end{proposition}

\begin{proposition}[Theorem 5.1 in \cite{MS1}]
Suppose $h_1 = [(x_1, x_2, x_3)] \in H_2(X)$ and $h_2 = [(y_1, y_2, y_3)] \in H_2(X, \partial X)$. Then, we have 
\[\langle h_1, h_2 \rangle_{(X, \partial X)} = -\sum_{1 \le i < j \le 3} \langle x_i, y_j \rangle_{\Sigma}, \]
where $\langle \cdot , \cdot \rangle_{(X, \partial X)}$ is the intersection form on $H_2(X) \times H_2(X, \partial X)$.
\end{proposition}

\begin{proposition}
The homology of $\partial X$ can be calculated by the following complex:
\vspace{-2ex}
\begin{figure}[H]
\begin{tikzcd}
0 \ar{r} & \mathbb{Z} \ar{r}{0} & H_1(\Sigma_1 , \partial \Sigma_1) \ar{r}{\xi_{\psi}} & H_1(\Sigma_1) \ar{r}{0} & \mathbb{Z} \ar{r} & 0.
\end{tikzcd}
\end{figure}
\end{proposition}

Lemma 4.2 still holds in this situation. Let $f_i$ be the following element of $L_1$;
\[ f_i = \begin{cases} -[a_{2k}]    & ( i \le 2p, \, i = 2k-1) \\
                           [a_{2k-1}] & ( i \le 2p, \, i = 2k) \\
                              0.         & ( 2p < i)
          \end{cases}\]
We take the following elements of $L_1 \oplus L_2 \oplus L_3$ and $L^{\partial}_1 \oplus L^{\partial}_2 \oplus L^{\partial}_3$;
\[  \mu_i = ( [\alpha_i], \sum_j (\alpha \cdot \gamma)_{ji} [\beta_j], [\gamma_i]), \, (l < i) \] 
\[\nu_i = ( [\alpha_i] + f_i, \sum_j (\alpha \cdot \gamma)_{ji} [\beta_j], [\gamma_i]).\]
We can grant that $\mu_i$ are in $\ker \iota$ and $\ker \iota^{\partial}$, and $\mu_i = \nu_i$. Because $[\eta_{2k}] = [a_{2k-1}]$ and $[\eta_{2k-1}] = -[a_{2k}]$ for $k \le p$, and $[\eta_i] = 0$ otherwise in $H_1(\Sigma, \partial \Sigma)$, $\nu_i$ are in $\ker \iota^{\partial}$. Let $i_* : H_2(\partial X) \to H_2(X)$, $j_* : H_2(X) \to H_2(X, \partial X)$ and $\partial_* : H_2(X, \partial X) \to H_1(\partial X)$ be the homomorphisms derived from the exact sequence of $(X, \partial X)$.

\begin{lemma}
The collection $\{[\mu_{l+1}], \dots , [\mu_{l+b_2}] \}$ forms a basis of $H_2(X)$, and $\{[\nu_1], \dots , [\nu_{l+b_2}]\}$ generates $H_2(X, \partial X)$. Moreover, we have $j_*([\mu_i]) = [\nu_i]$, and $\partial_*([\nu_i]) = - d_i [\eta_i] \in H_1(\partial X)$.
\end{lemma}
\begin{proof}
First generativity follows from the Lemma 4.2 and independence of $\{[\alpha_i], [\beta_i], [\eta_j] \}$. It is clear that $\{\mu_i\}$ are linearly independent in $\ker \iota$. From these facts and $L_1 \cap L_2 = L_2 \cap L_3 = 0$, first statement follows. 

We can choose $\{[\alpha_i], [a_j]\}$, $\{[\beta_i], [a_j]\}$ and $\{[\gamma_i], [a^3_j]\}$ as bases of $L_1^{\partial}$, $L_2^{\partial}$ and $L_3^{\partial}$, respectively. Since $[a^3_j] = [a_j] - [\beta_j]$ and $\{[\alpha_i], [\beta_i], [a_j]\}$ is linearly independent in $H_1(\Sigma, \partial \Sigma)$, $\ker \iota^{\partial}$ can be generated by $\{\nu_i, \zeta^{\partial}(-[a_j], -[a^3_j], 0)\}$, and then $H_2(X, \partial X)$ can be generated by $\{\nu_i\}$. It is easy to see $j_*([\mu_i]) = [\nu_i]$.

For $i = 2k$, since the chain complex $C^Z_{\partial}$ is derived from CW-complex, or handle decomposition, we can regard $\nu_i$ as a linear combination of cells whose boundaries are $\alpha_i$, $\partial (a_{i-1} \times I)$, $\beta_j$ and $\gamma_i$. Let $\bar{X}^q$ be the $q$-skeleton of this CW-complex. Since $\iota^{\partial} (\nu_i) = 0$, we have $\partial_2 \nu_i = 0$ in $H_1(\bar{X}^1, \bar{X}^0)$, and then in $H_1(\bar{X}^1, \partial X)$ (by using the exact sequence of $(\bar{X}^1, \bar{X}^0, \partial X)$). $\partial_2 \nu_i$ can be represented by the singular chain $\alpha_i + \partial (a_{i-1} \times I) + \sum_j (\alpha \cdot \gamma)_{ji} \beta_j + \gamma_i$, and then this is reduced by $\partial (a_{i-1} \times I) - \eta_i$. Since $\partial_2 \nu_i = 0$, it can be modified to be in $\partial X$. Actually, we can choose $\partial (a_{i-1} \times D^1) - \eta_i$ as a representative because $a_{i-1}$ and $\eta_i$ are disjoint from $\alpha$, where $D^1$ is the arc of the base circle of the open book decomposition. Since $[\partial (a_{i-1} \times D^1)] = 0$ in $H_1(\partial X)$, we have $\partial_*([\nu_i]) = -[\eta_i]$. Other $i$ cases are followed in the same way.
\end{proof}
\noindent Note that when we treat $-[\eta_i]$ as an element of the homology calculated by the Proposition B.4, it can be regarded as the element represented by $-\eta_i \in \Sigma_1$.

We give a proof of another version of the theorem.

\begin{proof}
Let $a$ and $\eta$ be collections of curves in the assumption. In the same way as the proof of previous version, we can take a trisection diagram $(\Sigma; \alpha, \beta, \gamma)$ satisfying the conditions of Lemma 4.2. 

Since $H_2(\partial X) \cong H_1(\partial X) \cong \mathbb{Z}^r$, the following exact sequences derived the pair $(X, \partial X)$ are split;
\vspace{-2ex}
\begin{figure}[H]
\begin{tikzcd}
0 \ar{r} & H_2(\partial X) \ar{r}{i_*} & H_2(X) \ar{r}{j_*} & \im j_* \ar{r}{0} &  0,
\end{tikzcd}
\vspace{-1ex}
\end{figure}
\begin{figure}[H]
\begin{tikzcd}
0 \ar{r} & \im j_*  \ar{r} & H_2(X, \partial X) \ar{r}{\partial_*} & H_1(\partial X) \ar{r}{0} &  0.
\end{tikzcd}
\end{figure}
\noindent Note that $\{[a_{l-r+1}], \dots , [a_l]\}$ and $\{[\eta_{l-r+1}], \dots , [\eta_l]\}$ are bases of $H_2(\partial X)$ and $H_1(\partial X)$, respectively. Applying the previous argument regarding a basis of $\langle [a_{l-r+1}], \dots , [a_l] \rangle \oplus \im j_*$, we take a new diagram $(\Sigma; \alpha, \beta, \gamma)$ satisfying $[\gamma_{l+i}] = i_* ([a_{l-r+i}])$ for $1 \le i \le r$. At this time, $\{[\nu_{l-r+1}], \dots , [\nu_l], [\nu_{l+r+1}], \dots , [\nu_{l+b_2}]\}$ forms a basis of $H_2(X, \partial X)$ from Lemma B.5. Since $j_* \circ  i_* = 0$, we have $\langle [\gamma_{l+i}], x \rangle_X = 0$ for $1 \le i \le r$ and $x \in H_2(X)$. Therefore, we can represent the intersection matrix of $X$ as $\langle 0 \rangle^r \oplus Q$, where $Q$ is a $(b_2-r, b_2-r)$-matrix. 

Next, we consider the intersection form of $(X, \partial X)$. Using Proposition B.3, we have $\langle [\mu_i], [\nu_j] \rangle_{(X, \partial X)} = (\alpha \cdot \gamma)_{ji}$ for $l < i$. Since $\langle i_*(x), y \rangle_{(X, \partial X)} = \langle x, \partial_*(y) \rangle_{\partial X}$ for $x \in H_2(\partial X)$ and $y \in H_2(X, \partial X)$, its representation matrix regarding $\{[\mu_{l+1}], \dots , [\mu_{l+b_2}]\}$ and $\{[\nu_{l-r+1}], \dots , [\nu_l], [\nu_{l+r+1}], \dots , [\nu_{l+b_2}]\}$ can be partitioned into 4 blocks such that the upper left is the intersection matrix of $\partial X$, the lower left is the $(b_2-r, r)$-zero matrix, and the lower right is $Q$. Since the intersection form of $\partial X$ $(\cong \sharp^r S^1 \times S^2)$ and $(X, \partial X)$ are unimodular, so $Q$ is. Note that the intersection matrix of $\partial X$ is $I_r$ regarding the bases $\{[a_{l-r+1}], \dots [a_l]\}$ and $\{[\eta_{l-r+1}], \dots , [\eta_l]\}$. Since $[\nu_i] = 0$ for $i \le l-r$, $(\alpha \cdot \gamma)_{ij} = 0$ for $l  < j$. Using Lemma 4.3, we can modify $\alpha \cdot \gamma$ to the following form;
\[\alpha \cdot \gamma = \begin{pmatrix} B & D \\ {}^tD & O_r \end{pmatrix} \oplus Q \oplus \langle 0 \rangle^{k-l}, \]
where ${}^tD = \begin{pmatrix} O_{r, l-r} & I_r \end{pmatrix}$. Moreover, since the matrix $\alpha \cdot \gamma$, and so $B$, is almost symmetry, we can modify it to the objective form by using the previous argument regarding $I_r$, that is, we have $\alpha \cdot \gamma = B \oplus \begin{pmatrix} O & I_r \\ I_r & O \end{pmatrix} \oplus Q \oplus \langle 0 \rangle^{k-l}$. We can prove $B = A^{-1}$ in the same way as the last part of the proof.
\end{proof}

We consider open book decompositions of $\sharp^r S^1 \times S^2$. Generally, the following theorem regarding open book decomposition:

\begin{proposition}[Theorem 3.5 \cite{PZ1}]
Any two open book decompositions of a closed, connected, oriented 3-manifold $M$ can be related by isotopy after applying  (positive or negative) stabilizations, and the special case of move $\partial U$.
\end{proposition}
\noindent The $\partial U$-move is derived from the $U$-move for an achiral Lefschetz fibration $X$ whose boundary is $M$. Moreover, since this $U$-move is local and similar to the stabilization, it can be described by relative trisection diagram. Therefore, the following proposition holds:
\begin{proposition}
Any two relative trisection of $X$ can be made isotopic a finite number of interior stabilizations, (positive or negative) relative stabilizations and $U$-moves.
\end{proposition}

\noindent It is known that $\sharp^r S^1 \times S^2$ has the unique standard tight contact structure. Therefore, it is sufficient to consider open book decompositions supporting this tight structure. Regarding the assumption of the theorem, the following question arises:

\begin{question*}
For any open book decompositions of $\sharp^r S^1 \times S^2$ supporting the standard tight contact structure, does it exist a modification by a finite number of Hopf stabilizations satisfying the assumption? What if the $\partial U$-moves are allowed?
\end{question*}

\noindent If an affirmative answer is given to this question, we can represent cork twist for a simply connected 4-manifold as cutting and regluing regarding a relative trisection diagram by an element of the Torelli group. That is to say, we can restate the following famous theorem by using relative trisection.

\begin{proposition}[\cite{Ma1}, \cite{CFHS1}, \cite{AM1}]
For any simply connected closed exotic pair of 4-manifolds $X$ and $Y$, there exist a contractible 2-handlebody $C$, an embedding $C \to X$ and an involution $\tau : \partial C \to \partial C$ such that the cork twist $X(C, \tau)$ is $Y$.
\end{proposition}

\noindent For these $X$ and $Y$, $X \setminus C = Y \setminus C$ can be regarded as a handlebody with no 0-handle, $r$ 3-handles and one 4-handle. Therefore, $X \setminus \natural^r S^1 \times D^3$ and $Y \setminus \natural^r S^1 \times D^3$ are 2-handlebodies with boundary $\sharp^r S^1 \times S^2$. They have $(g, l; p, b)$-trisections $\mathcal{T}_X$ and $\mathcal{T}_Y$ which induce the same open book decomposition because of Proposition B.6. Moreover, if the Question is solved with affirmative answer, we can take this decomposition such that it satisfies the assumption. In the same way as the proof of the previous corollaries, the following conjecture is true.

\begin{conjecture*}
Let $X$, $Y$ be an exotic pair of smooth, oriented, simply connected, closed of 4-manifolds. Then, there are embeddings $\natural^r S^1 \times D^3 \to X$ and $\natural^r S^1 \times D^3 \to Y$,  $(g, l; p, b)$-trisection diagram $(\Sigma; \alpha, \beta, \gamma)$ of $X \setminus \natural^r S^1 \times D^3$, and an element $\rho \in \mathcal{I}(\Sigma)$ such that $(\Sigma; \alpha, \beta, \rho(\gamma))$ is a diagram of $Y \setminus \natural^r S^1 \times D^3$.
\end{conjecture*} 

Whether the Question is affirmative or not, open book decomposition (or contact structure) and the Torelli group  may play a more important role in the 4-dimensional topology than before.

\end{document}